\documentclass[11pt,twoside]{amsart}

\usepackage[
  paper=a4paper,
  headsep=15pt,text={155mm,230mm},centering
]{geometry}

\usepackage[bookmarks]{hyperref}
\usepackage{amssymb,amsxtra}
\usepackage{graphicx,xcolor}
\usepackage{float,array,calc,booktabs,stackrel,url}
\usepackage{enumitem}\setlist[enumerate,1]{font=\upshape}
\usepackage{mathtools}
\usepackage{tikz}
\usepackage{todonotes}
\usetikzlibrary{
  cd,
  calc,
  positioning,
  arrows,
  decorations.pathreplacing,
  decorations.markings,
}
\usepackage[mode=buildnew]{standalone}

\iftrue
    
    \makeatletter
    \def\@settitle{%
      \vspace*{-10pt}
      \begin{flushleft}%
        \LARGE\bfseries
        \strut\@title\strut
      \end{flushleft}%
    }
    \def\@setauthors{%
      \begingroup
      \def\thanks{\protect\thanks@warning}%
      \trivlist
      \raggedright
      \large \@topsep27\p@\relax
      \advance\@topsep by -\baselineskip
    \item\relax
      \author@andify\authors
      \def\\{\protect\linebreak}%
      \authors
      \ifx\@empty\contribs
      \else
      ,\penalty-3 \space \@setcontribs
      \@closetoccontribs
      \fi
      \normalfont
      \endtrivlist
      \endgroup
    }
    \def\@setaddresses{\par
      \nobreak \begingroup
      \small\raggedright
      \def\author##1{\nobreak\addvspace\smallskipamount}%
      \def\\{\unskip, \ignorespaces}%
      \interlinepenalty\@M
      \def\address##1##2{\begingroup
        \par\addvspace\bigskipamount\noindent
        \@ifnotempty{##1}{(\ignorespaces##1\unskip) }%
        {\ignorespaces##2}\par\endgroup}%
      \def\curraddr##1##2{\begingroup
        \@ifnotempty{##2}{\nobreak\noindent\curraddrname
          \@ifnotempty{##1}{, \ignorespaces##1\unskip}\/:\space
          ##2\par}\endgroup}%
      \def\email##1##2{\begingroup
        \@ifnotempty{##2}{\nobreak\noindent E-mail address%
          \@ifnotempty{##1}{, \ignorespaces##1\unskip}\/:\space
          \ttfamily##2\par}\endgroup}%
      \def\urladdr##1##2{\begingroup
        \def~{\char`\~}%
        \@ifnotempty{##2}{\nobreak\noindent\urladdrname
          \@ifnotempty{##1}{, \ignorespaces##1\unskip}\/:\space
          \ttfamily##2\par}\endgroup}%
      \addresses
      \endgroup
      \global\let\addresses=\@empty
    }
    \def\@setabstracta{%
      \ifvoid\abstractbox
      \else
      \skip@17pt \advance\skip@-\lastskip
      \advance\skip@-\baselineskip \vskip\skip@
      \box\abstractbox
      \prevdepth\z@ 
      \vskip-28pt
      \fi
    }
    \renewenvironment{abstract}{%
      \ifx\maketitle\relax
      \ClassWarning{\@classname}{Abstract should precede
        \protect\maketitle\space in AMS document classes; reported}%
      \fi
      \global\setbox\abstractbox=\vtop \bgroup
      \normalfont\small
      \list{}{\labelwidth\z@
        \leftmargin0pc \rightmargin\leftmargin
        \listparindent\normalparindent \itemindent\z@
        \parsep\z@ \@plus\p@
        
      }%
    \item[\hskip\labelsep\bfseries\abstractname.]%
    }{%
      \endlist\egroup
      \ifx\@setabstract\relax \@setabstracta \fi
    }

    \def\ps@headings{\ps@empty
      \def\@evenhead{%
        \setTrue{runhead}%
        \normalfont\scriptsize
        \rlap{\thepage}\hfill
        \def\thanks{\protect\thanks@warning}%
        \leftmark{}{}}%
      \def\@oddhead{%
        \setTrue{runhead}%
        \normalfont\scriptsize
        \def\thanks{\protect\thanks@warning}%
        \rightmark{}{}\hfill \llap{\thepage}}%
      \let\@mkboth\markboth
    }\ps@headings

    \def\section{\@startsection{section}{1}%
      \z@{-1.4\linespacing\@plus-.5\linespacing}{.8\linespacing}%
      {\normalfont\bfseries\Large}}
    \def\subsection{\@startsection{subsection}{2}%
      \z@{-.8\linespacing\@plus-.3\linespacing}{.5\linespacing\@plus.2\linespacing}%
      {\normalfont\bfseries\large}}
    \def\subsubsection{\@startsection{subsubsection}{3}%
      \z@{.7\linespacing\@plus.2\linespacing}{-1.5ex}%
      {\normalfont\bfseries}}
    \def\paragraph{\@startsection{paragraph}{4}%
      \z@{.7\linespacing\@plus.2\linespacing}{-1.5ex}%
      {\normalfont\itshape}}
    \def\@secnumfont{\bfseries}

    \renewcommand\contentsnamefont{\bfseries}
    \def\@starttoc#1#2{\begingroup
      \setTrue{#1}%
      \par\removelastskip\vskip\z@skip
      \@startsection{}\@M\z@{\linespacing\@plus\linespacing}%
      {.5\linespacing}{
        \contentsnamefont}{#2}%
      \ifx\contentsname#2%
      \else \addcontentsline{toc}{section}{#2}\fi
      \makeatletter
      \@input{\jobname.#1}%
      \if@filesw
      \@xp\newwrite\csname tf@#1\endcsname
      \immediate\@xp\openout\csname tf@#1\endcsname \jobname.#1\relax
      \fi
      \global\@nobreakfalse \endgroup
      \addvspace{32\p@\@plus14\p@}%
      \let\tableofcontents\rela\x
    }
    \def\contentsname{Contents}
    \def\l@section{\@tocline{2}{.5ex}{0mm}{5pc}{}}
    \def\l@subsection{\@tocline{2}{0pt}{2em}{5pc}{}}
    \makeatother

\fi

\def\to{\mathchoice{\longrightarrow}{\rightarrow}{\rightarrow}{\rightarrow}}
\makeatletter
\newcommand{\shortxra}[2][]{\ext@arrow 0359\rightarrowfill@{#1}{#2}}
\def\longrightarrowfill@{\arrowfill@\relbar\relbar\longrightarrow}
\newcommand{\longxra}[2][]{\ext@arrow 0359\longrightarrowfill@{#1}{#2}}

\makeatother

\makeatletter
\def\addtagsub#1{\let\oldtf=\tagform@\def\tagform@##1{\oldtf{##1}\hbox{$_{#1}$}}}
\makeatother

\makeatletter
\def\Nopagebreak{\@nobreaktrue\nopagebreak}
\makeatother

\makeatletter
\newtheoremstyle{theorem-giventitle}
        {}{}              
        {\itshape}                      
        {}                              
        {\bfseries}                     
        {.}                             
        {\thm@headsep}                             
        {\thmnote{\bfseries#3}}
\newtheoremstyle{theorem-givenlabel}
        {}{}              
        {\itshape}                      
        {}                              
        {\bfseries}                     
        {.}                             
        {\thm@headsep}                             
        {\thmname{#1}~\thmnumber{#3}\setcurrentlabel{#3}}
\newtheoremstyle{definition-giventitle}
        {}{}              
        {}                      
        {}                              
        {\bfseries}                     
        {.}                             
        {\thm@headsep}                             
        {\thmnote{\bfseries#3}}
\def\setcurrentlabel#1{\gdef\@currentlabel{#1}}
\makeatother

\newtheorem{theorem}{Theorem}[section]
\newtheorem{theoremalpha}{Theorem}
\newtheorem{corollaryalpha}[theoremalpha]{Corollary}

\newtheorem{lemma}[theorem]{Lemma}

\theoremstyle{definition}

\newtheorem{example}[theorem]{Example}
\newtheorem{remark}[theorem]{Remark}

\newtheorem*{case2'}{Case 2$'$}

\theoremstyle{theorem-giventitle}
\newtheorem{theorem-named}{}
\theoremstyle{theorem-givenlabel}
\newtheorem{theorem-labeled}{Theorem}

\theoremstyle{definition-giventitle}
\newtheorem{definition-named}{}
\newtheorem{conjecture-named}{}
\newtheorem{case-named}{}

\numberwithin{equation}{section}

\def\Z{\mathbb{Z}}

\def\Q{\mathbb{Q}}
\def\C{\mathbb{C}}

\DeclareMathOperator\Hom{Hom}

\def\Hom{\mathrm{Hom}}
\def\SL{\mathrm{SL}}
\def\GL{\mathrm{GL}}
\def\tr{\mathrm{tr}}
\def\irr{\mathrm{irr}}
\def\red{\mathrm{red}}
\def\M{\mathrm{M}}
\def\la{\langle}
\def\ra{\rangle}

\makeatletter

\begin{document}

\title[Quantitative finiteness of monic characters of knots]{Quantitative finiteness of monic characters of knots}


\author{Taehee Kim}
\address{Department of Mathematics\\
  Konkuk University\\
  Seoul 05029\\
  Republic of Korea 
}
\email{tkim@konkuk.ac.kr}

\author{Takayuki Morifuji}
\address{Department of Mathematics, Hiyoshi Campus, Keio University,
	Yokohama 223-8521, Japan}
\email{morifuji@keio.jp}

\thanks{The first-named author was supported by the National Research Foundation of Korea (NRF) grant funded by the Korea government (MSIT) (No.~2023R1A2C1003418). The second-named author has been supported by JSPS KAKENHI Grant Number JP25K07012.}

\def\subjclassname{\textup{2010} Mathematics Subject Classification}
\expandafter\let\csname subjclassname@1991\endcsname=\subjclassname
\expandafter\let\csname subjclassname@2000\endcsname=\subjclassname


\begin{abstract} 
Dunfield, Friedl, and Jackson showed that if the $\SL(2, \C)$-character variety of a knot has an irreducible curve component that contains the character of an irreducible representation and a character with nonmonic twisted Alexander polynomial, then this component has only finitely many characters with monic twisted Alexander polynomials. In this paper, we give explicit upper bounds on the number of such characters in terms of a presentation of the knot group. In particular, we give upper bounds in terms of the crossing number of the knot.  
\end{abstract}
\maketitle


\section{Introduction}\label{section:introduction}
The twisted Alexander polynomial is a generalization of the Alexander polynomial, and it was introduced by Wada \cite{Wad94} for finitely presentable groups and Lin \cite{Lin01} for knots in the 3-sphere. 
It has been applied to a variety of problems in knot theory and 3-manifold topology, including the genus of a knot \cite{FK06}, the Thurston norm of a 3-manifold \cite{FK08, FV15}, and the fiberedness of knots and 3-manifolds \cite{Cha03, GKM05-1, FK06, FV11-1}. 
Notably, Friedl and Vidussi showed that the twisted Alexander polynomials associated with finite representations determine fiberedness of knots and 3-manifolds \cite{FV11-1} and detect the Thurston norm of irreducible 3-manifolds that are not closed graph manifolds \cite{FV15}.

In this paper, we study the detection of fiberedness of a knot using twisted Alexander polynomials associated with the characters of $\SL(2,\C)$-representations.  
Let $K$ be a knot in~$S^3$, and let $X(K)$ be the $\SL(2,\C)$-character variety of $K$ (see Section~\ref{section:character variety} for the definition of the character variety of $K$). 
For $\chi\in X(K)$, let $\Delta_{K, \chi}(t)$ be the twisted Alexander polynomial associated with $\chi$ (see Section~\ref{section:TAP} for the definition). 
If $\chi$ is the character of a nonabelian $\SL(2,\C)$-representation, then $\Delta_{K,\chi}(t)$ is a polynomial with complex coefficients \cite{KM05}. 
In this case, a character $\chi$ is called \emph{monic} if the coefficient of the highest-degree term of $\Delta_{K,\chi}(t)$ is~$1\in \C$. 

In \cite{GKM05-1}, it was shown that if a knot is fibered, i.e., the complement $S^3\setminus K$ fibers over the circle, then every character of a nonabelian $\SL(2,\C)$-representation of the knot is monic; it is conjectured that the converse holds.
For hyperbolic knots, Dunfield--Friedl--Jackson \cite[Conjecture~1.4]{DFJ12} formulated the seemingly stronger conjecture that if the character of the distinguished $\SL(2,\C)$-representation of a hyperbolic knot is monic, then the knot is fibered.

In this direction, a natural question is how effective the twisted Alexander polynomials associated with $\SL(2,\C)$-characters are for detecting fiberedness. 
It is known that if~$X_0$ is an irreducible curve component of $X(K)$ that contains the character of an irreducible $\SL(2,\C)$-representation and a nonmonic character, then there are only finitely many monic characters in $X_0$: this was shown for twist knots \cite{Mor08} , for 2-bridge knots \cite{KKM13}, and later for all knots \cite{DFJ12}.

In this paper, we make the finiteness result effective 
by giving explicit upper bounds on the number of monic characters. In particular, we give an upper bound in terms of the crossing number of a knot. 

To state our main result, we need the following term: for a group $G$ with generators $x_i$, $1\le i\le n$, a word on the generators $x_i$ is said to be \emph{positive} if it is written without using $x_i^{-1}$ for any $1\le i \le n$. 

The main result of this paper is the following. 

\begin{theoremalpha}\label{theorem:main theorem presentation}
Let $K$ be a knot in $S^3$ and $G=\pi_1(S^3\setminus K)$.
Suppose that $G$ admits a presentation $G = \langle x_1,\ldots, x_n\mid r_1,\ldots,r_{n-1}\rangle$ in which the relations $r_i$ are written as $u_i=v_i$ with the lengths of $u_i$ and $v_i$ at most $\ell\ge 2$.
Suppose that $X_0$ is an irreducible curve component of $X(K)$ that contains the character of an irreducible $\SL(2,\C)$-representation of $G$ and a nonmonic character. 
Let $\mathcal{N}$ be the number of monic characters in $X_0$.
Then, the following hold.
\begin{enumerate}
	\item We have 
	\[
	\mathcal{N}\le (n-1)2^{n+2}\ell^{3n-3}.
	\] 
	Furthermore, if the words $u_i$ and $v_i$ are positive for all $i$, then we have 
	\[
	\mathcal{N}\le (n-1)(\ell-1)2^{n+2}\ell^{3n-4}.
	\]
	\item Suppose additionally that $x_1,\ldots, x_n$ are mutually conjugate. Then, we have 
	\[\mathcal{N}\le (n-1)2^{n+2}\ell^{2n-2}.
	\] 
	Furthermore, if the words $u_i$ and $v_i$ are positive for all $i$, then we have
	\[
	\mathcal{N}\le (n-1)(\ell-1)2^{n+2}\ell^{2n-3}.
	\]
\end{enumerate}	
\end{theoremalpha}
Theorem~\ref{theorem:main theorem presentation} gives a quantitative refinement of the qualitative finiteness result in \cite{DFJ12} with explicit bounds depending on the presentation of the knot group.
The main idea is to pull back the relevant level sets of the leading coefficient function of the twisted Alexander polynomials from $X_0$ to an irreducible component $R_0$ of the $\SL(2,\C)$-representation variety and estimate their degrees using B\'{e}zout's inequality.
Fox calculus \cite{Fox53} provides bounds for the degrees of the pull-backs of the level sets, while the defining equations of the representation variety give a bound for the degree of $R_0$. 
In Theorem~\ref{theorem:main theorem presentation}(2), the equal-trace condition arising from the mutually conjugate generators reduces the dimension of the ambient algebraic set and gives sharper bounds.

We note that hyperbolic knots \cite{CS83} and torus knots \cite{MO10, Mun09} have an irreducible curve component $X_0$ in $X(K)$ that contains the character of an irreducible representation. For examples of knots with such $X_0$ with nonmonic characters, see \cite{KKM13}.

A knot with crossing number $n$ admits a Wirtinger presentation with relations $r_i\colon x_i x_j=x_j x_k$, and hence $u_i$ and $v_i$ are positive and $\ell =2$. 
Since the generators in a Wirtinger presentation are mutually conjugate, Theorem~\ref{theorem:main theorem presentation}(2) gives the following corollary.

\begin{corollaryalpha}\label{corollary:main theorem crossing number}
	Let $K$ be a knot in $S^3$.
	Suppose that $X(K)$ has an irreducible curve component~$X_0$ that contains the character of an irreducible $\SL(2,\C)$-representation and a nonmonic character.
	If $K$ has crossing number~$n$, then the number of monic characters in $X_0$ is at most $(n-1)2^{3n-1}$. 
\end{corollaryalpha}

This paper is organized as follows. In Section~\ref{section:character variety} we review the $\SL(2,\C)$-representation variety of a knot and its character variety. 
In Section~\ref{section:TAP}, we review the twisted Alexander polynomial of a knot and its properties. 
We give a proof of Theorem~\ref{theorem:main theorem presentation} in Section~\ref{section:bounds}.

\subsection*{Acknowledgments} 
We thank Jae Choon Cha for suggesting this project to us and Stefan Friedl for helpful discussions at an early stage of this work. We also thank Insong Choe for helpful conversations. 

\subsection*{Tool and computational resource disclosure}
ChatGPT (GPT-5.6 Sol and GPT-6 Astra) was used for literature searches, language editing, and assistance in developing, refining, and checking some mathematical arguments. An initial explicit upper bound was obtained by the authors using Lemma~\ref{lemma:Bezout inequality} (B\'{e}zout's inequality). GPT-5.6 Sol subsequently suggested using a refined version of B\'{e}zout's inequality, which we formulated as Lemma~\ref{lemma:Bezout lemma} and used to improve the bound. All mathematical statements, arguments, and references were independently verified by the authors, who take full responsibility for the content of this paper. 


\section{The character variety}\label{section:character variety}
In this section, we briefly review the representation variety and character variety of a knot~$K$ following Culler and Shalen~\cite{CS83}. Let $G=\pi_1(S^3\setminus K)$, the knot group of $K$. The {\it representation variety} of $K$ is $R(K) = \Hom(G, \SL(2,\C))$. 
A representation $\rho \in R(K)$ is \emph{abelian} if $\rho(G)$ is abelian, and \emph{nonabelian} otherwise. 
A representation $\rho \in R(K)$ is \emph{reducible} if there exists a proper $\rho(G)$-invariant subspace of~$\C^2$, and \emph{irreducible} otherwise.

It is well known that $R(K)$ is an affine algebraic set~\cite[Section~1.1]{CS83}; suppose that $G$ has a presentation $G = \langle x_1,\ldots, x_n\mid r_1,\ldots,r_{n-1}\rangle$. 
Each $\rho\in R(K)$ is determined by its values 
\begin{equation}\label{equation:matrix of rho}
\rho(x_i)=
\begin{pmatrix}
		a_i(\rho) & b_i(\rho)\\
		c_i(\rho) & d_i(\rho)
\end{pmatrix}
\in \SL(2, \C), 
\quad 1\le i\le n.
\end{equation}
Hence, $R(K)$ is embedded as an affine algebraic set in $\SL(2,\C)^n\subset \C^{4n}$ with coordinates $a_i, b_i, c_i, d_i$, $1\le i\le n$, defined by $n$ equations 
\[
\det \begin{pmatrix}
	a_i & b_i\\
	c_i & d_i
\end{pmatrix}=1, \quad 1\le i\le n,
\] 
and $4(n-1)$ equations obtained from the relations $r_i$. 
 
The \emph{character} of $\rho\in R(K)$ is a map $\chi_\rho\colon G\to \C$ defined by $\chi_\rho(g) = \tr(\rho(g))$, and the \emph{character variety} $X(K)$ is the set of characters of representations in $R(K)$.

For $g\in G$, we define a function $\tau_g\colon R(K)\to \C$ by $\tau_g(\rho) =\tr (\rho(g))=\chi_\rho(g)$. 
Let $T$ be the ring generated by $\tau_g$ for all $g\in G$. By \cite[Proposition~1.4.1]{CS83}, $T$ is finitely generated.
Let $g_1,\ldots, g_m$ be elements in $G$ such that $\tau_{g_i}$, $1\le i\le m$, generate $T$, and define a regular map $\tau\colon R(K)\to \C^m$ by $\tau=(\tau_{g_1},\ldots, \tau_{g_m})$. Then, there is a one-to-one correspondence between $X(K)$ and $\tau(R(K))$; throughout the paper we identify $X(K)$ with $\tau(R(K))$. 

The character variety $X(K)$ is an affine algebraic set in~$\C^m$ \cite[Corollary~1.4.5]{CS83}.
We note that $\SL(2, \C)$ acts on $\C^{4n}$ by conjugation: we have 
\[
\C^{4n} = \M(2,\C)^n=\left\{\left(A_1, \ldots, A_n\right)\mid A_i\in \M(2, \C), 1\le i\le n\right\},
\]
and $P\in \SL(2, \C)$ acts by 
\[
P\cdot \left(A_1, \ldots, A_n\right) = \left(PA_1P^{-1}, \ldots, PA_nP^{-1}\right).
\]
The conjugation action of $\SL(2, \C)$ on $R(K)$ induces an action on the coordinate ring $\C[R(K)]$, and the coordinate ring $\C[X(K)]$ is $\C[R(K)]^{\SL(2,\C)}$, the ring of regular functions on $R(K)$ that are invariant under the conjugation action of $\SL(2,\C)$.

If $\chi_\rho=\chi_{\rho'}$ and $\rho$ is irreducible, then $\rho'$ is irreducible \cite[Proposition~1.5.2]{CS83}; we say that a character $\chi_\rho\in X(K)$ is \emph{irreducible} (respectively, \emph{reducible}) if $\rho\in R(K)$ is irreducible (respectively, reducible). 

Let $R^\irr(K)$ and $R^\red(K)$ denote the set of irreducible $\rho\in R(K)$ and the set of reducible $\rho\in R(K)$, respectively. 
Let $X^\irr(K)$ denote the Zariski closure of $\tau(R^\irr(K))$ and $X^\red(K)$ denote $\tau(R^\red(K))$. Then, $X^\irr(K)$ and $X^\red(K)$ are affine algebraic sets, and $X^\red(K)$ is the same as the set of characters of abelian $\rho\in R(K)$ and has dimension one over $\C$. It is known that a subset of $X^\irr(K)$ is an irreducible component of $X^\irr(K)$ if and only if it is an irreducible component of $X(K)$ that contains an irreducible character.


\section{Twisted Alexander polynomials}\label{section:TAP}

In this section, we define the twisted Alexander polynomial of a knot following Wada \cite{Wad94}, and briefly review its properties.
For a knot $K$ in $S^3$, recall that $G=\pi_1(S^3\setminus K)$. Suppose that~$G$ has a presentation $G = \langle x_1,\ldots, x_n\mid r_1,\ldots,r_{n-1}\rangle$. 
Let $\alpha\colon G\to \Z=\la t \ra$ be a surjective homomorphism. 
For $\rho\in R(K)$, we have a group homomorphism $\rho\otimes \alpha\colon G\to \GL(2, \C[t^{\pm 1}])$, and it extends to a ring homomorphism $\Phi\colon \Z[G] \to \M(2, \C[t^{\pm 1}])$, 
where $\M(k,\C[t^{\pm 1}])$ denotes the ring of $k\times k$ matrices with entries in $\C[t^{\pm 1}]$.

Let $M$ be the $(n-1)\times n$ matrix whose $(i,j)$-entry is $\frac{\partial r_i}{\partial x_j}\in \Z[G]$, where $\frac{\partial\phantom{x}}{\partial x_j}$ denotes the Fox derivative \cite{Fox53}. 
Let $M_k$ be the $(n-1)\times (n-1)$ matrix obtained by deleting the $k$th column of $M$ for $1\le k\le n$, and let $\Phi(M_k)$ denote the $2(n-1)\times 2(n-1)$ matrix with entries in $\C[t^{\pm 1}]$ obtained from $M_k$ by replacing each entry $m_{ij}$ by $\Phi(m_{ij})$. 
That is, $\Phi(M_k)\in M(2(n-1), \C[t^{\pm 1}])$. 
For $k$ such that $\alpha(x_k)\ne 1$, we define the \emph{twisted Alexander polynomial $\Delta_{K,\rho}(t)\in \C(t)$ associated with $\rho\in R(K)$} by 
\begin{equation}\label{equation:TAP}
\Delta_{K,\rho}(t) = \frac{\det\Phi(M_k)}{\det \Phi(x_k-1)},
\end{equation}
which is well-defined up to multiplication by $t^{2i}$, $i\in \Z$ and independent of the choice of such $k$. If a representation $\rho\in R(K)$ is nonabelian, then $\Delta_{K,\rho}(t)$ is a (Laurent) polynomial, i.e. $\Delta_{K,\rho}(t)\in \C[t^{\pm 1}]$  \cite[Theorem~3.1]{KM05}.

For $\rho, \rho'\in R(K)$, it is known that if $\chi_\rho = \chi_{\rho'}$, then $\Delta_{K,\rho}(t) = \Delta_{K,\rho'}(t)$ \cite[Section~2.2]{KKM13}. 
Thus, we can define the \emph{twisted Alexander polynomial $\Delta_{K,\chi}(t)$ of $K$ associated with $\chi\in X(K)$} by $\Delta_{K,\chi}(t) =\Delta_{K,\rho}(t)$ for any $\rho\in R(K)$ such that $\chi=\chi_\rho$. 

If $\chi\in X^\irr(K)$ is an irreducible character, then it is the character of an irreducible and hence nonabelian $\rho\in R(K)$. 
If $\chi\in X^\irr(K)$ is a reducible character, then it is the character of a nonabelian reducible $\rho\in R(K)$ by \cite[Section~6]{CCG94}.
 Therefore, if $\chi\in X^\irr(K)$, then it is the character of a nonabelian representation, and we have $\Delta_{K,\chi}(t)\in \C[t^{\pm 1}]$ by \cite[Theorem~3.1]{KM05}. 
In this case, we say that the character $\chi$ is \emph{monic} if $\Delta_{K,\chi}(t)$ is monic, that is, its coefficient of the highest-degree term is one.

Let $X_0$ be an irreducible component of $X^\irr(K)$.
By \cite[Proposition~1.4.4]{CS83}, there exists an irreducible component $R_0$ of $R(K)$ such that $\tau(R_0)=X_0$.
As shown in the proof of \cite[Theorem~1.5]{DFJ12}, for some $N$ there exist regular functions $\phi_i\colon R_0\to \C, ~ -N\le i\le N$ such that for all $\rho\in R_0$
\begin{equation}\label{equation:TAP-rep}
	\Delta_{K,\rho}(t) = \sum_{i=-N}^N\phi_i(\rho)t^i
\end{equation}
without any ambiguity. Moreover, these functions $\phi_i$ descend to regular functions on $X_0$.
Thus, there exist regular functions $\psi_i$ on $X_0$ such that for $\chi\in X_0$ we have
\begin{equation}\label{equation:TAP-chi}
	\Delta_{K,\chi}(t) = \sum_{i=-N}^N\psi_i(\chi)t^i.
\end{equation}
We call $\psi_m$ the \emph{leading-coefficient function of $\Delta_{K,\chi}(t)$ on $X_0$} if $\psi_m\not\equiv 0$ and $\psi_k\equiv 0$ for all $k>m$ on $X_0$.


\section{Upper bounds on the number of monic characters}\label{section:bounds}

For an irreducible affine algebraic set, we define its \emph{degree} to be the degree of its projective closure, and for a reducible affine algebraic set we define its \emph{degree} to be the sum of the degrees of its irreducible components; here irreducible components of all dimensions are included in the sum. 
We denote the degree of an affine algebraic set $V$ by $\deg(V)$.  
Recall that if $f$ is a nonzero nonconstant polynomial on $\C^n$ and $V(f)$ is the set of zeros of $f$ in 
$\C^n$, then 
\[
\deg(V(f))\le \deg(f).
\]

We will use the following lemmas.
\begin{lemma}[{\cite[Theorem~1]{Hei83}, \cite{Hei85}}](B\'{e}zout's inequality)\label{lemma:Bezout inequality}
If $X$ and $Y$ are affine algebraic sets in $\C^n$, then $\deg(X\cap Y)\le \deg(X)\cdot \deg(Y)$.
\end{lemma}
\begin{lemma}\label{lemma:Bezout lemma}
Let $V$ be an affine algebraic set in $\C^n$. Let $d\ge 1$, and let $f_i$ be a polynomial on~$\C^n$ of degree at most $d$, $1\le i\le m$. Let $Z$ be an irreducible component of $V\cap V(f_1,\ldots, f_m)$ and let $c=\dim(V) - \dim(Z)$. Then, we have
\begin{equation}\label{equation:Bezout lemma}
\deg(Z)\le \deg(V)\cdot d^c.
\end{equation}
\end{lemma}

\begin{proof}
First suppose that $V$ is irreducible, and set $V_0=V$. 
If all $f_i$ are identically zero on~$V_0$, then $V_0\cap V(f_1,\ldots, f_m)=V_0$ and hence $Z\subset V_0$. 
This implies that $Z=V_0$ since $V_0$ is irreducible and~$Z$ is an irreducible component of $V_0$. 
Then, $\deg(Z) = \deg(V_0)$ and \eqref{equation:Bezout lemma} holds.

Otherwise, choose $f_{i_1}$ that is not identically zero on $V_0$. 
Let $V_1$ be an irreducible component of $V_0\cap V(f_{i_1})$ that contains $Z$. 
Then, $Z$ is an irreducible component of $V_1\cap V(f_1,\ldots, f_m)$.
By Krull's principal ideal theorem,
\[
\dim(V_1)= \dim(V_0)-1
\] 
since $V_0$ is irreducible and $f_{i_1}$ is not identically zero on~$V_0$.

We iterate this process and obtain irreducible affine algebraic sets $V_1,\ldots,V_c$ and polynomials $f_{i_1},\ldots, f_{i_c}$ such that $f_{i_k}~(1\le k\le c)$ is not identically zero on $V_{k-1}$, and $V_k$ is an irreducible component of $V_{k-1}\cap V(f_{i_k})$ that contains $Z$, and $\dim(V_k)= \dim(V_{k-1})-1$. 
Then, we have $Z\subset V_c$ and $\dim Z = \dim V_c$. It follows that $Z=V_c$ since $Z$ and $V_c$ are irreducible. 

By Lemma~\ref{lemma:Bezout inequality}, we have 
\[
\deg(V_k)\le \deg(V_{k-1}\cap V(f_{i_k}))\le \deg (V_{k-1})\cdot \deg(V(f_{i_k}))\le \deg (V_{k-1})
\cdot d.
\]
Therefore, we have $\deg(Z)\le \deg(V)\cdot d^c$. 

Now suppose that $V$ is reducible. Let $V_0$ be an irreducible component of $V$ that contains~$Z$. 
Then $Z$ is an irreducible component of $V_0\cap V(f_1,\ldots, f_m)$ and as above we have
\[
\deg(Z)\le \deg(V_0)\cdot d^{\dim(V_0)-\dim(Z)}.
\]
Since $\deg(V_0)\le \deg(V)$ and $\dim(V_0)\le \dim(V)$, we have
\[
\deg(V_0)\cdot d^{\dim(V_0)-\dim(Z)}\le \deg(V)\cdot d^c,
\]
and we obtain \eqref{equation:Bezout lemma}.
\end{proof}

Now, we give a proof of Theorem~\ref{theorem:main theorem presentation}.
\begin{proof}[{\bf Proof of Theorem~\ref{theorem:main theorem presentation}}] 

For a set $S$, let $\#S$ denote the cardinality of $S$.
Let~$\psi$ be the leading-coefficient function of $\Delta_{K,\chi}(t)$ on $X_0$.
In particular, $\psi\not\equiv  0$ on $X_0$, and by assumption $\psi\not\equiv 1$ on $X_0$.
For $z\in \C$, let $P_z = \{\chi\in X_0\mid \psi(\chi)=z\}$. The sets $P_0$ and $P_1$ are finite since they are proper Zariski closed subsets of a curve $X_0$. 
Then, we have 
\begin{equation}\label{equation:P_0 and P_1}
\{\chi\in X_0\mid \chi \text{ is monic}\}\, \subset\, P_0\cup P_1.
\end{equation}

\medskip

\noindent {\bf Proof of~(1).}
By \eqref{equation:P_0 and P_1}, it suffices to show that for each $z=0,1$, we have 
\[
\#P_z\le (n-1)2^{n+1}\ell^{3n-3},
\] 
and with an additional assumption that $u_i$ and $v_i$ are positive for all $i$, we have 
\[
\#P_z\le (n-1)(\ell-1)2^{n+1}\ell^{3n-4}.
\]

Let $\phi=\psi\circ (\tau|_{R_0})\colon R_0\to \C$, where $\tau\colon R(K)\to \C^m$ is the regular map introduced in Section~\ref{section:character variety}. 
Then, $\phi$ is the leading-coefficient function of $\Delta_{K,\rho}(t)$ in~\eqref{equation:TAP-rep} on~$R_0$. 
We need the following lemma.
\begin{lemma}\label{lemma:regular extension to the ambient space}
	There exists a polynomial map $\overline{\phi}\colon \C^{4n}\to \C$ that restricts to $\phi$ on $R_0$ and has $\deg(\overline{\phi})\le 2(n-1)\ell$. Furthermore, if the words $u_i$ and $v_i$ on $x_1,\ldots, x_n$ are positive, then $\deg(\overline{\phi})\le 2(n-1)(\ell-1)$.
\end{lemma}
We postpone the proof of Lemma~\ref{lemma:regular extension to the ambient space} to the end of Proof of~(1).

Fix $z\in \{0,1\}$. 
The set $P_z\subset X_0$ is finite, and we let $P_z=\{p_1,\ldots,p_r\}$ and let \[
Y=(\tau|_{R_0})^{-1}(P_z).
\] 

Since $P_z$ has $r$ connected components and $\tau|_Y\colon Y\to P_z$ is continuous and surjective, $Y$ has at least $r$ connected components with respect to the Euclidean topology.
Every irreducible component of a complex affine algebraic set is Euclidean connected (see \cite[VII, Section~2.2]{Sha77}); hence $Y$ has at least $r$ irreducible components.
Recall that $\deg(Y)$ is the sum of the degrees of irreducible components of $Y$. 
Since the degree of an irreducible component is at least one, it follows that $\#P_z = r\,\le\, \deg (Y)$. 

One can readily see that 
\[Y=(\tau|_{R_0})^{-1}(P_z)=R_0\cap V(\overline{\phi}-z),
\] 
where $V(\overline{\phi}-z)$ is the zero set of the polynomial $\overline{\phi}-z$ in $\C^{4n}$. 
By Lemma~\ref{lemma:Bezout inequality}, we have 
\[
\deg(Y)\,= \,\deg(R_0\cap V(\overline{\phi}-z))\,\le\, \deg(R_0)\cdot \deg(V(\overline{\phi}-z)),
\]
and $\deg(V(\overline{\phi}-z)) \le \deg(\overline{\phi}-z)=\deg(\overline{\phi})$, 
where the equality holds since $\overline{\phi}\not\equiv z$.

By Lemma~\ref{lemma:regular extension to the ambient space}, it suffices to show 
\begin{equation}\label{equation:R_0}
\deg(R_0)\le 2^n\ell^{3n-4}.
\end{equation}
Recall that $\C^{4n}$ is equipped with coordinate $a_1,b_1,c_1,d_1,\ldots, a_n,b_n,c_n,d_n$. Recall that $R_0\subset R(K)\subset \C^{4n}$ and $R(K)$ is defined by $n$ equations
\[
\det \begin{pmatrix}
	a_i & b_i\\
	c_i & d_i
\end{pmatrix}=1, \quad 1\le i\le n,
\] 
and $4(n-1)$ equations obtained from the relations $r_i$, $1\le i\le n-1$. We explain these $4(n-1)$ equations in detail.

Each relation $r_i$ has the form $u_i=v_i$, where $u_i$ and $v_i$ are words on $x_1^{\pm 1},\ldots, x_n^{\pm 1}$ of lengths at most $\ell$. 
For $1\le k\le n$, let $A_k$ and $A_k^*$ be $2\times 2$ matrices defined by 
\[
A_k=\begin{pmatrix}
	a_k & b_k\\
	c_k & d_k
\end{pmatrix},
\quad
A_k^*=\begin{pmatrix}
	d_k & -b_k\\
	-c_k & a_k
\end{pmatrix}.
\]
We note that if $\rho\in R(K)\subset \SL(2, \C)^n$, we have
\[
A_k(\rho)^*=\begin{pmatrix}
	d_k(\rho) & -b_k(\rho)\\
	-c_k(\rho) & a_k(\rho)
\end{pmatrix}=
A_k(\rho)^{-1}.
\]
Let $u_i(A)$ (respectively, $v_i(A)$) denote the $2\times 2$ matrix with entries in 
\[
\C[\C^{4n}] =\C[a_1, b_1, c_1, d_1, \ldots, a_n, b_n, c_n, d_n]
\] 
obtained by replacing $x_k$ and $x_k^{-1}$ in $u_i$ (respectively, $v_i$) by $A_k$ and $A_k^*$ for all~$k$. 
Then, from four entries of each $u_i(A)-v_i(A)$, $1\le i\le n-1$, we obtain four polynomials in $a_1, b_1, c_1, d_1, \ldots, a_n, b_n, c_n, d_n$. 
We remark that these $4(n-1)$ polynomials have degrees  at most $\ell$. 
Since $\ell\ge 2$, all the above equations defining $R(K)$, including the determinant equations, have degrees at most $\ell$.

Recall that $R(K)\subset \SL(2,\C)^n$.  
Let 
\[
f_i= a_id_i-b_ic_i-1, \quad1\le i\le n, 
\]
and let $f_{n+1},\ldots, f_{5n-4}$ be the $4(n-1)$ equations obtained from the relations $r_i$, $1\le i\le n-1$. 
Then, $\SL(2, \C)^n= V(f_1,\ldots, f_n)$ and $R(K) = V(f_1,\ldots, f_{5n-4})$.
Since $\deg(f_i)\le 2$, $1\le i\le n$, by Lemma~\ref{lemma:Bezout inequality} we have
\[
\deg (\SL(2,\C)^n)  \le \deg(V(f_1))\cdots \deg(V(f_n)) 
				  \le \deg(f_1)\cdots \deg(f_n)
				  \le  2^n.
				  \]

Note that $\deg(f_i)\le \ell$, $n+1\le i\le 5n-4$, and $R_0\subset R(K) = \SL(2,\C)^n\cap V(f_{n+1},\ldots, f_{5n-4})$. Moreover, we have $\dim (\SL(2, \C))=3$, and hence $\dim (\SL(2,\C)^n)=3n$.
Since $\dim X_0=1$, it follows from \cite[Corollary~1.5.3]{CS83} that $\dim R_0=4$. 
We apply Lemma~\ref{lemma:Bezout lemma} with $V=\SL(2,\C)^n$ and $Z=R_0$, and obtain the desired inequality 
\[
\deg(R_0)\le \deg(\SL(2,\C)^n)\cdot \ell^{3n-4} \le 2^n\ell^{3n-4}.
\]

We complete Proof of~(1) by giving the proof of Lemma~\ref{lemma:regular extension to the ambient space} below.
\begin{proof}[Proof of Lemma~\ref{lemma:regular extension to the ambient space}]
After relabeling the generators if necessary, we may assume that $\alpha(x_n)\ne~1$. From \eqref{equation:TAP}, 
we have 
\[
\Delta_{K,\rho}(t) = \frac{\det\Phi(M_n)}{\det \Phi(x_n-1)},
\]
and $\det \Phi(x_n-1) = t^{2\alpha(x_n)} - (a_n+d_n)t^{\alpha(x_n)} +1$. 
Since the denominator is monic, $\phi$ is therefore the leading-coefficient function of $\det \Phi(M_n)$ on $R_0$. 

Recall that $M_n=(m_{ij})\in \M(n-1, \Z[G])$ with $m_{ij}=\frac{\partial r_i}{\partial x_j}\in \Z[G], ~ 1\le i,\, j\le n-1$, 
and $\Phi(M_n) = (\Phi(m_{ij}))\in \M(2(n-1), \C[t^{\pm 1}])$.

The Fox derivative satisfies the following rules.

\begin{enumerate}
	\item[(R1)] $\displaystyle\frac{\partial x_i}{\partial x_j}=\delta_{ij}$, where $\delta_{ij}$ is the Kronecker delta.
	\item[(R2)] $\displaystyle\frac{\partial (uv)}{\partial x_i}=\frac{\partial u}{\partial x_i} + u\frac{\partial v}{\partial x_i}$ for $u, v\in G$.
	\item[(R3)] $\displaystyle\frac{\partial u^{-1}}{\partial x_i} = -u^{-1}\frac{\partial u}{\partial x_i}$.
\end{enumerate}
Each $r_i$ is a relation $u_i=v_i$ and as a relator it is $u_iv_i^{-1}$. An easy calculation shows that 
\[
m_{ij}=\frac{\partial (u_iv_i^{-1})}{\partial x_j} = \frac{\partial u_i}{\partial x_j}  - u_iv_i^{-1}\frac{\partial v_i}{\partial x_j}.
\]
The relator $u_iv_i^{-1}$ is trivial in $G$, and it follows that $\Phi(u_iv_i^{-1}) = I_2\in \M(2, \C[t^{\pm 1}])$. 
Thus, we have
\[
\Phi(m_{ij}) = \Phi\left(\frac{\partial (u_iv_i^{-1})}{\partial x_j}\right) = \Phi\left(\frac{\partial u_i}{\partial x_j}\right) - \Phi\left(\frac{\partial v_i}{\partial x_j}\right)
= \Phi\left(\frac{\partial u_i}{\partial x_j} - \frac{\partial v_i}{\partial x_j}\right).
\]
It follows from (R1)--(R3) that each $\frac{\partial u_i}{\partial x_j} - \frac{\partial v_i}{\partial x_j}$ is a sum of words on $x_1^{\pm 1},\ldots, x_n^{\pm 1}$ of lengths at most $\ell$, and furthermore if $u_i$ and $v_i$ are positive, then the lengths are at most $\ell-1$.

Let $\left( \frac{\partial u_i}{\partial x_j} - \frac{\partial v_i}{\partial x_j} \right)(A)$ be the $2\times 2$ matrix obtained from $\frac{\partial u_i}{\partial x_j} - \frac{\partial v_i}{\partial x_j}$ by replacing $x_k$ and $x_k^{-1}$ in $\frac{\partial u_i}{\partial x_j} - \frac{\partial v_i}{\partial x_j}$ by $A_k t^{\alpha(x_k)}$ and $A_k^* t^{-\alpha(x_k)}$, respectively, for all $k$. 

Let $N\in \M(2(n-1), (\C[\C^{4n}])[t^{\pm 1}])$ be the matrix obtained from $M_n$ by replacing $m_{ij}$ by $\left( \frac{\partial u_i}{\partial x_j} - \frac{\partial v_i}{\partial x_j} \right)(A)$. 
Then, we have $\det(N)\in (\C[\C^{4n}])[t^{\pm 1}]$, a (Laurent) polynomial in $t$ with coefficients in $\C[\C ^{4n}]$. 
Note that each entry of $A_k$ and $A_k^*$ is a polynomial in $a_k, b_k, c_k, d_k$ of degree one. 
Therefore, each entry of $N$ is a polynomial in $a_k, b_k, c_k, d_k$, $1\le k\le n$, of degree at most $\ell$, and degree at most $\ell-1$ if $u_i$ and $v_i$ are positive. 
It follows that each coefficient of $\det(N)$ is a polynomial in $a_k, b_k, c_k, d_k$, $1\le k\le n$, of degree at most $2(n-1)\ell$, and degree at most $2(n-1)(\ell-1)$ if $u_i$ and $v_i$ are positive.

Let $\overline{\phi}$ be the coefficient of $\det (N)$ whose restriction to $R_0$ is the leading-coefficient function~$\phi$. 
It follows that  $\overline{\phi}$ is a polynomial map on $\C^{4n}$ that restricts to $\phi$ on $R_0$, and this completes the proof.
\end{proof}

\noindent
{\bf Proof of~(2).}
Using the assumption that all $x_i$ are mutually conjugate, we will show that 
\begin{equation}\label{equation:R_0-x_i-meridian}
\deg(R_0)\le 2^n\ell^{2n-3}.
\end{equation}
The proof of~(2) then follows from the proof of~(1), with \eqref{equation:R_0} replaced by \eqref{equation:R_0-x_i-meridian}.

From the assumption that the generators $x_i$ are mutually conjugate, it follows that for each $\rho\in R(K)$, we have $\chi_\rho(x_1)=\chi_\rho(x_i)$, $2\le i\le n$.
Define the regular function $s\colon R(K)\to \C$ by 
\[
s(\rho)=\chi_\rho(x_1)=a_1(\rho)+d_1(\rho).
\]
Then, for each $\rho\in R(K)$, we have
\[
\rho(x_i) = \begin{pmatrix}
	a_i(\rho) & b_i(\rho) \\
	c_i(\rho) & s(\rho)-a_i(\rho)
\end{pmatrix},
\quad
1\le i\le n.
\] 

Let $W=\{(A_1,\ldots, A_n)\in \SL(2,\C)^n\mid \tr(A_1)=\cdots = \tr(A_n)\}\subset \C^{4n}$. 
Then, $R(K)\subset W$. 
We will show that $\dim(W)=2n+1$ and $\deg(W)\le 2^n$. 
Let $s$, $a_1,b_1, c_1,\ldots, a_n, b_n, c_n$ be the coordinates on $\C^{3n+1}$ and let 
\[
h_i=a_i(s-a_i)-b_ic_i-1,\quad 1\le i\le n,
\]
be polynomials of degree 2 on $\C^{3n+1}$. 

Let $Z=V(h_1,\ldots, h_n)\subset \C^{3n+1}$. 
Let $L$ be the linear subspace of $\mathrm{M}(2,\C)^n=\C^{4n}$ defined by $\tr(A_1)=\cdots =\tr(A_n)$.
Define a linear map $g\colon \C^{3n+1}\to L$ by
\[
g(s,a_1,b_1, c_1,\ldots, a_n, b_n, c_n) = (a_1, b_1, c_1, s-a_1, \ldots, a_n, b_n, c_n, s-a_n).
\]
Then $g$ is a linear isomorphism and $g(Z)=W$.
It follows that 
\[
\dim(Z) =\dim(W) \text{ and }\deg(Z)=\deg(W).
\]

We show that $\dim(Z)=2n+1$. 
Let $Z_0$ be an irreducible component of $Z$. 
Since $Z\subset \C^{3n+1}$ is defined by $n$ equations, Krull's height theorem gives 
\[
\dim(Z_0)\ge 2n+1.
\]
Define the map $p\colon Z\to \C$ by 
\[
p(s,a_1,b_1, c_1,\ldots, a_n, b_n, c_n)=s.
\]
For $s_0\in \C$ we have 
\[
p^{-1}(s_0)=\{s_0\}\times \prod_{i=1}^n V\big(a_i(s_0-a_i)-b_ic_i-1\big)\subset \{s_0\}\times \C^{3n},
\]
and each $V\big(a_i(s_0-a_i)-b_ic_i-1\big)\subset \C^3$ has dimension 2, $1\le i\le n$. 
It follows that $\dim(p^{-1}(s_0))=2n$ for all $s_0\in \C$. 
Since $\dim(Z_0)\ge 2n+1$, the restriction map $p|_{Z_0}$ is not constant.
Thus, by Krull's principal ideal theorem, for $s_0\in p(Z_0)$ we have 
\[
\dim(Z_0\cap p^{-1}(s_0)) = \dim(Z_0)-1.
\]
Since $Z_0\cap p^{-1}(s_0)$ is a subset of $p^{-1}(s_0)$, we have
\[
\dim(Z_0)-1 \le \dim(p^{-1}(s_0)) = 2n.
\]
Therefore, we have $\dim(Z_0)\le 2n+1$, and hence $\dim(Z_0)= 2n+1$.
Recall that the dimension of an affine algebraic set is the maximum dimension of its irreducible components.
Since every irreducible component of $Z$ has dimension $2n+1$, we have $\dim(Z)=2n+1$; hence $\dim(W)=2n+1$.

By Lemma~\ref{lemma:Bezout inequality}, we have 
\[
\deg(Z)\le \prod_{i=1}^n\deg(V(h_i))\le \prod_{i=1}^n\deg(h_i)\le 2^n.
\]
Therefore, $\deg(W)=\deg(Z)\le 2^n$.

Recall that $f_{n+1},\ldots, f_{5n-4}$ are the polynomials  obtained from the relations $r_i$. Thus, 
we have
\[
R_0\subset R(K)=W\cap V(f_{n+1},\ldots, f_{5n-4}).
\]
We apply Lemma~\ref{lemma:Bezout lemma} with $V=W$ and $Z=R_0$, and obtain the desired inequality 
\[
\deg(R_0)\le \deg(W)\cdot \ell^{\dim(W)-4}\le 2^n\ell^{2n-3}.\qedhere
\]
\end{proof} 

\begin{remark}\label{remark:same traces}
The proof above shows that the assumption in Theorem~\ref{theorem:main theorem presentation}(2) that $x_1,\ldots, x_n$ are mutually conjugate can be replaced by the weaker assumption that  
\[
\chi(x_1)=\cdots =\chi(x_n)
\]
for every $\chi\in X_0$.
\end{remark}


\begin{example}
Let $K=K(p,q)$ be a $2$-bridge knot, where $p,q$ are coprime odd integers such that $p>0$ and $-p<q<p$. Then, the character variety $X(K)$ is of dimension one \cite{Riley84-1}, and hence every irreducible component of $X(K)$ with positive dimension is a curve. Let $X_0$ be an irreducible component of $X(K)$ that contains the character of  an irreducible $\SL(2,\C)$-representation of $G=\pi_1(S^3\setminus K(p,q))$ and a nonmonic character. 

Since the crossing number of $K$ is less than or equal to $p$, we have 
\begin{equation}\label{equation:2-bridge}
\mathcal{N}=\# \{\chi\in X_0\mid \text{$\chi$ is monic}\}\le (p-1)2^{3p-1}
\end{equation}
by Corollary~\ref{corollary:main theorem crossing number}.

On the other hand, the knot group $G$ admits the presentation $G=\la x,y\,|\,wx=yw\ra$ where $w=x^{\epsilon_1}y^{\epsilon_2}\cdots x^{\epsilon_{p-2}}y^{\epsilon_{p-1}}$ and $\epsilon_i=(-1)^{[\frac{q}{p}i]},~1\leq i\leq p-1$. Since $x,y$ are conjugate, $n=2$, and $\ell=p$, we have 
\[
\mathcal{N}
\le (2-1)2^{2+2}p^{2\cdot2-2} = 16p^2
\] 
by Theorem~\ref{theorem:main theorem presentation}(2). 
This upper bound is sharper than (\ref{equation:2-bridge}) when $p\ge 3$. 

For a double twist knot $K(p,q)=J(k,l)$ where $\frac{q}{p}$ is equal to $\frac{l}{1-kl}$ in $\Q/\Z$ and $kl$ is even, 
a sharper estimate is given in \cite[Theorem 5.1]{KM12-1}.
\end{example}


\end{document}